\documentclass[12pt,leqno,twoside]{article}
\usepackage{amssymb}
\usepackage{amsmath}
\usepackage{amsthm}
\usepackage{a4,indentfirst,latexsym}
\usepackage{mathrsfs}
\usepackage{cite,enumitem,graphicx}
\usepackage{enumerate}
\usepackage{verbatim}
\usepackage[pdfborder={0 0 0}]{hyperref}
\usepackage[all]{xy}
\usepackage{csquotes}

\newtheorem{Thm}{Theorem}[section]
\newtheorem{Prop}[Thm]{Proposition}
\newtheorem{Lem}[Thm]{Lemma}

\newtheorem{Def}[Thm]{Definition}
\newtheorem{Rem}[Thm]{Remark}
\newtheorem{Ex}[Thm]{Example}

\newenvironment{altproof}[1]
{\noindent
{\em Proof of {#1}}.}
{\nopagebreak\mbox{}\hfill $\Box$\par\addvspace{0.5cm}}

\newcommand\mytop[2]{\genfrac{}{}{0pt}{}{#1}{#2}}
\newcommand\set[1]{\left\{\,#1\,\right\}}  
\newcommand\abs[1]{\left|#1\right|} 
\newcommand{\ska}[1]{\left\langle#1\right\rangle} 
\newcommand\norm[1]{\left\Vert#1\right\Vert} 
\newcommand\Sym{\text {Sym}}
\DeclareMathOperator{\Kern}{Kern}
\DeclareMathOperator{\Range}{Range}

\newcommand{\eps}{\varepsilon}

\def\deg{\mathrm{deg}}
\def\sdeg{\text{\rm $S^1$-deg$^\nabla$}}
\def\span{\mathrm{span}}

\def\opint{\mathrm{int}}
\def\clos{\mathrm{clos}}

\def\id{\mathrm{id}}

\def\Z{\mathbb{Z}}

\def\N{\mathbb{N}}
\def\R{\mathbb{R}}
\def\C{\mathbb{C}}

\def\sign{\mathrm{sign}}

\def\polhk#1{\setbox0=\hbox{#1}{\ooalign{\hidewidth
   \lower1.5ex\hbox{`}\hidewidth\crcr\unhbox0}}}

\newcommand{\cB}{{\mathcal B}}
\newcommand{\cC}{{\mathcal C}}
\newcommand{\cD}{{\mathcal D}}

\newcommand{\cF}{{\mathcal F}}

\newcommand{\cO}{{\mathcal O}}

\newcommand{\cS}{{\mathcal S}}

\newcommand{\cU}{{\mathcal U}}

\newcommand{\fJ}{{\mathfrak J}}
\newcommand{\fK}{{\mathfrak K}}

\newcommand{\al}{\alpha}

\newcommand{\ga}{\gamma}
\newcommand{\de}{\delta}

\newcommand{\om}{\omega}
\newcommand{\si}{\sigma}
\newcommand{\Ga}{\Gamma}
\newcommand{\De}{\Delta}
\newcommand{\La}{\Lambda}
\newcommand{\Om}{\Omega}
\newcommand{\Si}{\Sigma}

\newcommand{\pa}{\partial}
\def\id{\mathrm{id}}

\newcommand{\wh}{\widehat}

\newcommand{\ov}{\overline}

\newcommand{\beq[1]}{\begin{equation}\label{eq:#1}}
\newcommand{\eeq}{\end{equation}}

\numberwithin{equation}{section}
\DeclareMathOperator{\dist}{dist}
\begin{document}

\title{Global continua of periodic solutions of singular first-order Hamiltonian systems of N-vortex type}
\author{Thomas Bartsch \and Bj\"orn Gebhard}
\date{}
\maketitle

\begin{abstract}
The paper deals with singular first order Hamiltonian systems of the form
\[
\Ga_k\dot{z}_k(t)=J\nabla_{z_k} H\big(z(t)\big),\quad z_k(t)\in\Om\subset\R^2,\ k=1,\dots,N,
\]
where $J\in\R^{2\times2}$ defines the standard symplectic structure in $\R^2$, and the Hamiltonian $H$ is of $N$-vortex type:
\[
H(z_1,\dots,z_N)
 = -\frac1{2\pi}\sum_{\mytop{j,k=1}{j\ne k}}^N \Ga_j\Ga_k\log\abs{z_j-z_k} - F(z).
\]
This is defined on the configuration space $\{(z_1,\ldots,z_N)\in \Om^{2N}:z_j\neq z_k\text{ for }j\neq k\}$ of $N$ different points in the domain $\Om\subset\R^2$. The function $F:\Om^N\to\R$ may have additional singularities near the boundary of $\Om^N$. We prove the existence of a global continuum of periodic solutions $z(t)=(z_1(t),\dots,z_N(t))\in\Om^N$ that emanates, after introducing a suitable singular limit scaling, from a relative equilibrium $Z(t)\in\R^{2N}$ of the $N$-vortex problem in the whole plane (where $F=0$). Examples for $Z$ include Thomson's vortex configurations, or equilateral triangle solutions. The domain $\Om$ need not be simply connected. A special feature is that the associated action integral is not defined on an open subset of the space of $2\pi$-periodic $H^{1/2}$ functions, the natural form domain for first order Hamiltonian systems. This is a consequence of the singular character of the Hamiltonian. Our main tool in the proof is a degree for $S^1$-equivariant gradient maps that we adapt to this class of potential operators.
\end{abstract}



\section{Introduction}\label{sec:intro}
We consider Hamiltonian systems
\beq[Om1]
\Ga_k\dot{z}_k = J\nabla_{z_k}H(z),\quad k=1,\ldots,N,
\eeq
for $N$ point vortices $z_1(t),\dots,z_N(t)$ in a domain $\Om\subset\R^2$. Here
$
J=\begin{pmatrix}
0&1\\
-1 & 0
\end{pmatrix}
$
is the standard symplectic matrix in $\R^2$, and $\Ga_1,\dots,\Ga_N\in\R\setminus\{0\}$ are fixed vortex strengths. The Hamiltonian is singular and of the form
\beq[ham]
H(z)=-\frac{1}{2\pi}\sum_{\mytop{j,k=1}{j\ne k}}^N\Ga_j\Ga_k\log\abs{z_j-z_k} - F(z)
\eeq
with $F:\Om^N\to\R$ of class $\cC^2$. $H$ is defined on the configuration space
\[
\cF_N(\Om)=\{(z_1,\ldots,z_N)\in \Om^N:z_j\neq z_k\text{ for }j\neq k\}.
\]

Hamiltonian systems of this form appear in a variety of singular limit problems from mathematical physics. The classical point vortex problem from fluid dynamics goes back to Kirchhoff \cite{Kirchhoff:1876}. In the fluid dynamics context equation \eqref{eq:Om1} is derived from the Euler equations for an ideal fluid in $\Om$ when the vorticity is concentrated in vortex blobs $B_\de(z_k)$, $k=1,\dots,N$, and one passes to the singular limit $\de\to0$. Kirchhoff considered the case of the plane $\Om=\R^2$ and derived the Hamiltonian
\[
H_{0}(z)=-\frac{1}{2\pi}\sum_{\mytop{j,k=1}{j\ne k}}^N\Ga_j\Ga_k\log\abs{z_j-z_k},
\]
often called the Kirchhoff-Onsager functional. If $\Om\ne\R^2$ boundary effects play a role and the regular part $g:\Om\times\Om\to\R$ of a hydrodynamic Green's function (see \cite{Crowdy-Marshall:2005,Flucher-Gustafsson:1997}) in $\Om$ enters into the definition of the Hamiltonian:
\[
F(z)=\sum_{j,k=1}^N\Ga_j\Ga_kg(z_j,z_k).
\]
This has been derived by Routh \cite{Routh:1881} and C.C.~Lin \cite{Lin:1941_1,Lin:1941_2}, the Hamiltonian is then called Kirchhoff-Routh path function. For modern presentations of the point vortex method in fluid dynamics we refer to \cite{Marchioro-Pulvirenti:1994, Majda-Bertozzi:2001, Newton:2001, Saffmann:1995}. We would like to mention that in the present paper we allow more general nonlinearities $F$ which is relevant also for other applications.

Another motivation for considering \eqref{eq:Om1} arises in models of superconductivity. There one considers functions $u^\eps:\Om\times(0,\infty)\to\C$ solving the Ginzburg-Landau-Schr\"odinger (or Gross-Pitaevskii) equation
\[
iu^\eps_t-\De u^\eps =  \frac{1}{\eps^2} u^\eps(1-|u^\eps|^2).
\]
In this context a vortex is an isolated zero of $u^\eps$. In the limit $\eps\to0$ these vortices move according to \eqref{eq:Om1} provided the associated Ginzburg-Landau energy remains small. The number $\Ga_k$ is, up to a multiple, the Brouwer index of the zero $z_k$ of $u^\eps(\,.\,,t)$. In this context the Hamiltonian is the renormalized energy defined in \cite{Bethuel-etal:1994}; see \cite{Colliander-Jerrard:1998,Colliander-Jerrard:1999,Lin-Xin:1998} for more details. For problems on surfaces see \cite{Chen-Sternberg:2014,Gelantalis-Sternberg:2012}.

Still another motivation is the Landau-Lifshitz-Gilbert equation
\[
\frac{\pa{\bf m}}{\pa t} + {\bf m}\times\left(\De{\bf m}-\frac{m_3}{\eps^2}{\bf e}_3\right)
 + \alpha_\eps{\bf m}\times\left({\bf m}\times\left(\De{\bf m}-\frac{m_3}{\eps^2}{\bf e}_3\right)\right) = 0
\]
modeling the dynamics of a magnetic vortex system in a thin ferromagnetic film. The magnetization is given by a normalized vector field ${\bf m}:\Om\times(0,\infty)\to S^2$; $\eps$ is a material constant, $\al_\eps>0$ is a dimensionless damping constant, and ${\bf e}_3=(0,0,1)$. The magnetic vorticity is given by
\[
\om({\bf m}) = \left\langle{\bf m},\frac{\pa{\bf m}}{\pa x_1}\times\frac{\pa{\bf m}}{\pa x_2}\right\rangle.
\]
In the limit $\eps\to0$ with $\al_\eps\log\frac1\eps\to0$ the motion of point vortices is again given by a singular Hamiltonian system of the form \eqref{eq:Om1} with a Hamiltonian as in \eqref{eq:ham}; see \cite{Kurzke-etal:2011} and the references cited therein.

Due to the significance of \eqref{eq:Om1}-\eqref{eq:ham} many authors have investigated its dynamics, in particular for $\Om=\R^2$ or $\Om=S^2$. For domains with non-empty boundary much less is known, except in special cases like $\Om$ being a half-plane or radially symmetric (disc, annulus). In these cases the Green's function is explicitly known. For a general domain even the existence of equilibria is difficult to prove or disprove. Recent results on equilibria can be found in \cite{Bartsch-Pistoia:2015, Bartsch-Pistoia-Weth:2010, delPino-Kowalczyk-Musso:2005, Kuhl:2015, Kuhl:2016}. We would like to mention that these results do not give any information on the dynamics near an equilibrium. In particular it is not known whether the generalized Weinstein-Moser theorem \cite{Bartsch:1997} can be applied in order to find periodic solutions near the equilibrium. Concerning periodic solutions of \eqref{eq:Om1} in a general domain with boundary the only result we are aware of deals with the case $\Ga_1=\dots=\Ga_N$. In \cite{Bartsch-Dai:2016} the existence of a family $z^{(r)}=\left(z^{(r)}_1,\dots,z^{(r)}_N\right)$, $0<r<r_0$,  of periodic solutions with period $T_r$ has been proved. All vortices $z^{(r)}_k(t)$ rotate around a point $a_r\in\Om$ that is close to a critical point of the Robin function $h(a)=g(a,a)$, and they lie approximately on the vertices of a regular $N$-gon of distance $r$ from $a_r$. After a suitable scaling they look like Thomson's vortex configurations.

In the present paper we continue our investigations on periodic solutions of \eqref{eq:Om1}-\eqref{eq:ham} and generalize and improve the result from \cite{Bartsch-Dai:2016} significantly. First of all we deal with general vorticities $\Ga_k$, in particular they may be different and may have different signs. We start with a periodic relative equilibrium solution
of the vortex problem in the plane, i.e.\ a solution of
\beq[R2]
\Ga_k\dot{z}_k = J\nabla_{z_k}H_0(z),\quad k=1,\ldots,N,
\eeq
that rotates with frequency $\om$ around the origin and keeps its shape. Such solutions are also called vortex crystals and have been investigated by many authors. We refer the reader to \cite{Aref-etal:2003,Lewis-Ratiu:1996} for explicit examples, Thomson's vortex configuration being one of the simplest and best known ones. Then we give a criterion so that $Z$ generates a family $z^{(r)}$ of solutions of \eqref{eq:Om1}-\eqref{eq:ham} that look like $Z$, after a suitable singular limit scaling. Moreover, we show that these solutions lie on a global continuum of periodic solutions. This requires different methods than those used in this context before.

Normalizing the period to $2\pi$ by introducing a parameter $r>0$ that corresponds to the period, the solutions will be obtained as critical points of the action integral
\[
\fJ(r,u) = \fJ_r(u) =
 \frac12\int_0^{2\pi}\sum_{k=1}^N\ska{\Ga_k\dot{u}_k,Ju_k}_{\R^2}\:dt - \int_0^{2\pi}H_r(u)\:dt.
\]
The form domain of the quadratic form
\[
Q(z) = \int_0^{2\pi}\sum_{k=1}^N\ska{\Ga_k\dot{u}_k,Ju_k}_{\R^2}\:dt
\]
is the space $H^{1/2}=H^{1/2}(\R/2\pi\Z,\R^{2N})$. However, the functional $\int_0^{2\pi}H_r(u)\:dt$ is not defined on that space because $H_r$ inherits the singular behavior from $H$, and because $H^{1/2}$ does not embed into the space of continuous functions. Therefore the condition $u_j(t)\ne u_k(t)$ for $j\ne k$ does not define an open subset of $H^{1/2}$. There are a few other papers on singular first order Hamiltonian systems, most notably \cite{Carminati-Sere-Tanaka:2006,Tanaka:1996}. However in these papers the Hamiltonian is a variation of the $N$-body Hamiltonian from celestial mechanics and has a very different type of singular behavior. Moreover assumptions of "strong force" type are made, so that the Palais-Smale condition holds.

In order to find critical points of $\fJ_r$ we shall not apply methods from critical point theory. This seems to be hopeless at the moment because we cannot control the behavior of the Hamiltonian near the boundary of $\cF_N(\Om)$. In fact, $H(z)$ may approach any value in $\R\cup\set{-\infty,+\infty}$ as $z\to\pa\cF_N(\Om)$. As a consequence we do not see any kind of linking structure that leads to Palais-Smale sequences. Moreover, the functional $\fJ_r$ does not satisfy the Palais-Smale condition. Therefore instead of variational arguments we develop a variation of the degree theory for $S^1$-equivariant potential operators due to Rybicki \cite{Rybicki:1994}. His extension of this degree to strongly indefinite functionals in \cite{Rybicki:2001} cannot be used here because for our singular Hamiltonians the action functional $\fJ$ is not defined on the form domain of the quadratic form $Q$. In fact, we shall work on $H^1$ instead of $H^{1/2}$. It would be very interesting to see whether Floer type methods can be applied. We believe that our equivariant degree is especially useful for singular first order Hamiltonian systems.

The paper is organized as follows. After stating our results in the next section we introduce our degree in Section~\ref{sec:deg}. The following sections \ref{sec:setting}-\ref{sec:proof} contain the proof of our main theorem, the heart of it being the calculation of the degree in section~\ref{sec:deg-comp}. Finally in the last section \ref{sec:ex} we present some concrete examples of vortex crystals for which our main theorem holds.

\section{Statement of results}\label{sec:results}
Let $\Ga_1,\ldots,\Ga_N\in\R\setminus\{0\}$ be given vorticities, $N\geq 2$, let $\Om\subset \R^2$ be a domain and let $g:\Om\times\Om\to\R$ be a symmetric $\cC^2$-function. We consider the $N$-vortex type Hamiltonian
\[
H_\Om(z)
 = -\frac{1}{2\pi}\sum_{\mytop{j,k=1}{j\ne k}}\Ga_j\Ga_k\log\abs{z_j-z_k}
     - \sum_{j,k=1}^N\Ga_j\Ga_kg(z_j,z_k).
\]
which is defined on $\cF_N(\Om):=\{(z_1,\ldots,z_N)\in \Om^N:z_j\neq z_k\text{ for }j\neq k\}$. If $g$ is the regular part of a hydrodynamic Green's function then we arrive at the classical $N$-vortex Hamiltonian in the domain $\Om$.

In order to write the equation in a more compact way we introduce the vorticity matrix
\[
M_\Ga =
\begin{pmatrix}
\Ga_1&&\\
&\ddots&\\
&& \Ga_N
\end{pmatrix}\otimes E_2
=\begin{pmatrix}
\Ga_1E_2&&\\
&\ddots&\\
&& \Ga_NE_2
\end{pmatrix}
\in\R^{2N\times 2N}
\]
and the symplectic matrix
$J_N = E_N\otimes J\in \R^{2N\times 2N}$, where $E_m\in\R^{m\times m}$ is the identity matrix. We want to find periodic solutions $z:\R\to\cF_N(\Om)$ of
\beq[Om]
M_\Ga\dot{z} = J_N\nabla H_\Om (z).
\eeq

Recall the definition $h(z)=g(z,z)$ of the ``Robin'' function $h:\Om\to\R$. A critical point $a\in\Om$ of $h$ is said to be stable if it is isolated and has non-vanishing Brouwer index, i.e.\ the Brouwer degree $\deg(\nabla h,B_\eps(a),0)\ne0$ for $\eps>0$ small. A periodic relative equilibrium solution of \eqref{eq:R2} with center of vorticity at $0$ has the form
\beq[rel-equilib]
Z(t)=e^{-\om J_Nt}z,\quad \om\in\R\setminus\{0\},\ z\in\cF_N(\R^2).
\eeq
Such a relative equilibrium $Z$ is called non-degenerate, if the linearized system
\beq[R2lin]
M_\Gamma\dot{w} = J_N\big(H_0''(Z(t))\big)w
\eeq
possesses exactly three linearly independent $\frac{2\pi}{\abs{\om}}$-periodic solutions. This is the minimal positive dimension due to the invariance of $H_0$ under translations and rotations.
Observe that $Z$ as in \eqref{eq:rel-equilib} is a non-degenerate $\frac{2\pi}{\abs{\om}}$-periodic equilibrium if and only if $Z_\om(t):=\sqrt{|\om|}Z(t/|\om|)$ is a non-degenerate $2\pi$-periodic equilibrium. We can therefore assume that $Z$ is $2\pi$-periodic, i.e. $\abs{\om}=1$.

We write $X = H^1(\R/2\pi\Z,\R^{2N})$ for the Hilbert space of $2\pi$-periodic absolutely continuous functions $u:\R\to\R^{2N}$ with (locally) square integrable derivative. The standard scalar product in $X$ is
\[
\ska{u,v}_X
 = \int_0^{2\pi}\ska{u(t),v(t)}_{\R^{2N}}+\ska{\dot{u}(t),\dot{v}(t)}_{\R^{2N}}\:dt.
\]
For $u\in X$ and $\theta\in S^1=\R/2\pi\Z$ we define $\theta*u\in X$ by $\theta*u(t) := u(t+\theta)$. This defines a continuous representation of the group $S^1$ on $X$. For $a\in\R^2$ we set $\wh{a}:=(a,\dots,a)\in\R^{2N}$. We also need the subspace $D:=\{\wh{a}:a\in\R^2\}\subset\R^{2N}\subset X$ and the orthogonal projection $P_D:X\to D$.

\begin{Thm}\label{thm:main1}
  Let $Z$ be a non-degenerate $2\pi$-periodic relative equilibrium solution of \eqref{eq:R2} with center of vorticity at $0$, and let $a_0\in\Om$ be a critical point of $h$. If $a_0$ is stable, and if the total vorticity $\sum_{k=1}^N\Ga_k\ne0$, then there exists a connected $S^1$-invariant set $\cC=\cC(a_0,Z)\subset\R^+\times X$ with the following properties.
  \begin{itemize}
  \item[a)] If $(r,u)\in\cC$ then $z(t):=\wh{a_0}+ru(t/r^2)$ is a $2\pi r^2$-periodic solution of \eqref{eq:Om}.
  \item[b)] There exists $r_0>0$ and an $S^1$-invariant neighborhood $\cU\subset(0,r_0]\times X$ of
     $(0,r_0]\times S^1*Z$ such that:
     \[
     (r_n,u_n)\in\cC\cap\cU,\ r_n\to0,\ u_n=P_D[u_n]+v_n
      \quad\Longrightarrow\quad r_n P_D[u_n]\to0,\ S^1*v_n\to S^1*Z.
     \]
  \item[c)] For every $r\in(0,r_0]$ there exists an element $(r,u^{(r)})\in\cC^-:=\cC\cap\cU$.
  \item[d)] For $\cC^+:=\clos\big(\cC\setminus\cC^-\big)$ at least one of the following holds:
     \begin{itemize}
     \item[(i)] $\cC^+$ is unbounded.
     \item[(ii)] There exist sequences $(r_n,u_n)\in\cC^+$ and $t_n\in[0,2\pi]$ with $r_n$ bounded away from $0$ and $\wh{a_0}+r_nu_n(t_n/r_n^2)\to\pa\cF_N(\Om)$.
     \item[(iii)] There exists a sequence $(r_n,u_n)\in\cC^+$ with $r_n\to0$ and $u_n=P_D[u_n]+v_n$ is such that $r_nP_D[u_n]$ is bounded away from $0$ or $S^1*v_n$ is bounded away from $S^1*Z$.
     \end{itemize}
  \item[e)] If $a_0$ is a non-degenerate critical point of $h$, then there exists a $\cC^1$ function $(0,r_0]\ni r\mapsto u^{(r)}\in X$ such that $\cC^-=S^1*\operatorname{Graph}(u^{(.)})$.
\end{itemize}
\end{Thm}

\begin{Rem}\label{rem:main1}\rm
  a) Clearly any non-degenerate critical point of $h$ and any isolated local minimum or maximum is stable. As a consequence of \cite[Theorem~I.4.6, Theorem~II.3.1]{Chang:1993} an isolated critical point $a$ of $h$ is stable if and only if it has non-vanishing critical groups $H_*(h^c,h^c\setminus\{a\})$. Here $c=h(a)$, $h^c=\{z\in\Om:h(z)\le c\}$ is the usual sublevel set, and $H_*$ denotes any kind of homology theory.

  b) In a bounded domain a hydrodynamic Robin function satisfies $h(z)\to\infty$ as $z\to\pa\Om$, hence the minimum is achieved. If $\Om$ is bounded and convex then $h$ is strictly convex and has a unique (local and global) minimum which is nondegenerate; see \cite[Theorem 3.1]{Caffarelli-Friedman:1985}. There are domains with an arbitrarily large number of critical points of $h$, even simply connected ones. In \cite{Micheletti-Pistoia:2014} it is proved for the Dirichlet Green function in a generic domain that critical points of $h$ are non-degenerate.

  c) Using a rotating coordinate frame it is easy to see that a periodic relative equilibrium $Z=e^{-\om J_Nt}z$ as in \eqref{eq:rel-equilib} is non-degenerate if and only if
  \beq[R2lin-rot]
  \dot{w} = J_N\left(M_\Ga^{-1}H_{0}''(z)+\om\cdot\id\right)w
  \eeq
  possesses exactly three linearly independent $\frac{2\pi}{\abs{\om}}$-periodic solutions.

  d) The local part of the theorem can be extended to the case where $\Om$ is an open subset of a two-dimensional surface. An extension of the global result that takes the topology of the surface into account is an interesting open problem.
\end{Rem}

\begin{Ex}\label{ex:2vortices}\rm
  Two vortices with vorticities $\Ga_1,\Ga_2\neq 0$ and such that $\Ga_1+\Ga_2\neq 0$ rotate rigidly around their center of vorticity. Such an equilibrium is always non degenerate; for details see Example~\ref{ex:N=2} below.
\end{Ex}

\begin{Ex}\label{ex:3vortices}\rm
  Three vortices with vorticities $\Ga_1,\Ga_2,\Ga_3\ne 0$ placed on the edges of an equilateral triangle form a relative equilibrium. It is non-degenerate provided the total vortex angular momentum $L=\Ga_1\Ga_2+\Ga_1\Ga_3+\Ga_2\Ga_3$ and the total vorticity $\Ga=\Ga_1+\Ga_2+\Ga_3$ satisfy
  \[
  \Ga\neq 0,\quad L\neq 0\quad \text{and}\quad L\neq \Ga_1^2+\Ga_2^2+\Ga_3^2.
  \]
  This will be proved in Example~\ref{ex:N=3}.
\end{Ex}

Notice that the conditions of Example~\ref{ex:3vortices} do not hold in the important special case of three identical vortices. In order to treat this case we need a refinement of our main theorem including symmetries. The symmetric group $\Si_N$ on $N$ symbols $\set{1,\dots,N}$ acts isometrically on $\R^{2N}$ via permutation of components, i.e.\
\[
\si*z = (z_{\si^{-1}(1)},\ldots,z_{\si^{-1}(N)}),\quad \si\in\Si_N, z\in\R^{2N}.
\]
Together with the action of $S^1$ on $X$ we obtain an action of $\Si_N\times S^1$ on $X$ given by
\[
(\si,\theta)*u(t) := (u_{\si^{-1}(1)}(\,\cdot\,+\theta),\ldots,u_{\si^{-1}(N)}(\,\cdot\,+\theta)),\quad
\theta\in S^1,\ \si\in\Si_N,\ u\in X.
\]
If some of the vorticities $\Ga_1,\ldots,\Ga_N$ are equal, then the Hamiltonians $H_{0}$ and $H_\Omega$ are invariant under the action of a subgroup of $\Si_N$. This additional symmetry can also be found in some solutions of \eqref{eq:R2} and \eqref{eq:Om}. We set
$$
\Sym(\Ga)=\set{\si\in\Si_N: M_\Ga(\si*z) = \si*(M_\Ga z)\text{ for all }z\in\R^{2N}}
$$
and consider the symmetry group
$$
\Sym(H)
 = \set{\ga=(\si,\theta)\in\Sym(\Ga)\times S^1:\theta\in\frac{2\pi}{\text{ord}(\si)}\Z}.
$$
Given $\ga\in\Sym(H)$ we look for solutions in the space
\[
X^\ga := \set{u\in X: \ga*u=u}.
\]

\begin{Def}\rm
  Let $\ga\in\Sym(H)$. A relative equilibrium solution $Z=e^{\pm J_Nt}z\in X^\ga$ of \eqref{eq:R2} as in \eqref{eq:rel-equilib} is said to be $\ga$-non-degenerate, if the space
  $\set{w\in X^\ga: w\text{ solves \eqref{eq:R2lin}}}$
  has dimension three.
\end{Def}

\begin{Ex}\label{ex:NGon}\rm
  Thomson's point vortex configuration, i.e.\ the relative equilibrium consisting of $N$ identical vortices placed at the edges of a regular $N$-gon, is a $\ga$-non-degenerate solution of \eqref{eq:R2}, where $\ga=((1~2~\ldots~N),\frac{2\pi}{N})\in\Si_N\times S^1$; see Example~\ref{ex:NGon-2} below.
\end{Ex}

\begin{Thm}\label{thm:main2}
  Let $\ga\in\Sym(H)$, let $Z\in X^\ga$ be a $\ga$-non-degenerate relative equilibrium of \eqref{eq:R2} with center of vorticity at $0$, and let $a_0\in\Omega$ be a critical point of the Robin function $h$. If $a_0$ is stable, and if the total vorticity $\sum_{k=1}^N\Ga_k\ne0$, then there exists an invariant continuum  $\cC_\ga=\cC_\ga(a_0,Z)\subset\R^+\times X^\ga$ of periodic solutions of \eqref{eq:Om} with the properties a)--e) of Theorem~\ref{thm:main1}.
\end{Thm}

Note that Theorem \ref{thm:main1} is a special case of Theorem \ref{thm:main2} with $\ga=(\id,0)\in\Sym(H)$.

\begin{Rem}\rm
  a) Theorem~\ref{thm:main2} generalizes and improves \cite[Theorem~2.1]{Bartsch-Dai:2016}. In that paper for $Z$ only the case of Thomson's regular $N$-gon was considered. Moreover, since variational methods instead of degree methods were applied no connected continuum was found and the result was only local. The assumption in \cite[Theorem~2.1]{Bartsch-Dai:2016} that the critical groups of $h$ at $a_0$ are nontrivial is equivalent to our assumption that the Brouwer index of $\nabla h$ at $a_0$ is non-trivial; see Remark~\ref{rem:main1}~a).

  b) A very interesting and challenging problem consists in desingularizing the periodic solutions obtained for the point vortex problem to regular solutions of the partial differential equations mentioned above. In \cite{Bartsch-Pistoia:2015,Cao-Liu-Wei:2013,Cao-Liu-Wei:2014} equilibria of \eqref{eq:Om} have been desingularized in order to obtain stationary solutions of the Euler equations for an ideal fluid. Concerning periodic solutions we are only aware of the paper \cite{Gelantalis-Sternberg:2012} where a special periodic relative equilibrium of the point vortex problem on the two-dimensional sphere $S^2$ was desingularized to  obtain rotating solutions of the Gross-Pitaevskii equation on $S^2$.
\end{Rem}

\section{Degree theory for equivariant potential operators}\label{sec:deg}
In this section we generalize the construction of the degree for $S^1$-equivariant potential operators due to Rybicki~\cite{Rybicki:1994, Rybicki:2001}; see also \cite{Balanov-Krawcewicz-Steinlein:2006, Ize-Vignoli:1999} for a homotopy-theoretic approach. We use the following notation, and refer to \cite{Adams:1969} for basic representation theory. If $S^1=\R/2\pi\Z$ acts on a space $X$ we write $\theta*u$ for the action of $\theta\in S^1$ on $u\in X$. Given a closed subgroup $K\subset S^1$ the set of fixed points under $K$ is denoted by
$X^K=\set{u\in X: \theta*u=u\text{ for all }\theta\in K}$. The isotropy group of $u\in X$ is denoted by $I_u=\set{\theta\in S^1: \theta*u=u}$. The irreducible real representation where $\theta\in S^1$ acts on $\R^2$ via multiplication with $\begin{pmatrix}\cos(k\theta) & -\sin(k\theta)\\ \sin(k\theta) & \cos(k\theta)\end{pmatrix}$ is denoted by $\R^2[k]$. In the sequel all representations of $S^1$ are equipped with a scalar product that is preserved by the action of $S^1$.
For $k\ge1$ let $V_k\subset V$ be the isotypical component corresponding to $\R^2[k]$, and let $V_0=V^{S^1}$ be the fixed point set of the action. Then $V\cong\bigoplus_{k=0}^\infty V_k$, and all but finitely many of the $V_k$ are trivial. Moreover, $V_j$ and $V_k$ are orthogonal for $j\ne k$. By Schur's lemma an equivariant linear map $L:V\to V$ maps each $V_k$ to itself; we denote the restriction by $L_k:=L|_{V_k}:V_k\to V_k$. Observe that for $k\ge1$ there is a complex structure on $V_k$ such that the action of $\theta\in S^1$ is given by multiplication with $e^{k\theta i}$. For $v\in V\setminus V_0$ let $\tau(v)\in V$ be the unit tangent vector to the orbit $S^1*v$ at $v$ such that $\langle\tau(v),\frac{d}{d\theta}\theta*v\rangle>0$. If $v\in V_0^\perp$ and using the complex structure this is just $i\cdot \frac{v}{\abs{v}}$. For $v\in V_0$ we set $\tau(v)=0$.

For the convenience of the reader we now recall the basic properties of the degree for $S^1$-equivariant gradient maps in the finite-dimensional setting. Let $\cC^k_{S^1,\nabla}$ be the class of maps $f:(\ov{\cO},\pa\cO)\to (V,V\setminus\{0\})$, defined on the closure of an open, bounded, $S^1$-invariant subset $\cO\subset V$ of some finite-dimensional orthogonal representation $V$ of $S^1$, such that $f=\nabla F$ is the gradient of an $S^1$-invariant $C^{k+1}$-function $F:\cD\to \R$ with $\ov{\cO}\subset \cD\subset V$ open and $S^1$-invariant. For $f\in\cC^0_{S^1,\nabla}$ there exists a degree
\[
\sdeg(f,\cO)
 = \left(d_k^\nabla(f,\cO)\right)_{k\in\N_0} \in \bigoplus_{k=0}^\infty \Z
\]
with the following properties, \cite{Rybicki:1994,Rybicki:1997}:
\begin{itemize}
  \item[(D1)] (Existence) If $d_k^\nabla(f,\cO)\ne0$ for some $k\in\N_0$ then there exists $v\in \cO\cap V^K$ with $f(v)=0$ where $K=S^1$ if $k=0$, resp.\ $K=\Z_k=\Z/k\Z$ if $k\ge1$.
  \item[(D2)] (Excision and additivity) If $f^{-1}(0)\subset \cO_1\cup\cO_2$ for two disjoint open $S^1$-invariant subsets $\cO_1,\cO_2\subset\cO$ then
      \[
      \sdeg(f,\cO) = \sdeg(f,\cO_1)+\sdeg(f,\cO_2).
      \]
  \item[(D3)] (Homotopy) Let $\cU\subset [0,1]\times V$ be open and bounded, and let $h:(\ov{\cU},\pa\cU)\to(V,V\setminus\{0\})$ be continuous. If $h_t=h(t,\,.\,):\cU_t=\set{v\in V:(t,v)\in\cU}\rightarrow V$ lies in $\cC^0_{S^1,\nabla}$ for each $t\in[0,1]$, then $\sdeg(h_t,\cU_t)$ is independent of $t\in[0,1]$.
  \item[(D4)] (Multiplicativity) If $f_i:(\ov{\cO_i},\pa\cO_i)\to(V_i,V_i\setminus\{0\})$, $i=1,2$, are in $\cC^0_{S^1,\nabla}$ then so is
      $f_1\times f_2:\big(\ov{\cO_1}\times\ov{\cO_2},\pa(\ov{\cO_1}\times\ov{\cO_2})\big)\to V_1\times V_2$, and
      \[
      d_k^\nabla(f_1\times f_2,\ov{\cO_1}\times\ov{\cO_2}) =
       \begin{cases}
         d_0^\nabla(f_1,\cO_1)\cdot d_0^\nabla(f_2,\cO_2) & \mbox{if } k=0; \\
         d_k^\nabla(f_1,\cO_1)\cdot d_0^\nabla(f_2,\cO_2)+d_0^\nabla(f_1,\cO_1)\cdot d_k^\nabla(f_2,\cO_2) & \mbox{if } k\ge1.
       \end{cases}
      \]
\end{itemize}
Now we formulate some explicit computations of the degree.
\begin{itemize}
  \item[(D5)] If $L:V\to V$ is a linear $S^1$-equivariant and symmetric isomorphism then is the degree $\sdeg(L,V):=\sdeg(L,B_1(0))$ given by:
   \[
   d_k^\nabla(L,V) = \begin{cases}
                    \sign\det L & \mbox{if } k=0, \\
                    \frac12\sign\det L\cdot\mu_k & \mbox{if $k\ne0$ and $V_k\ne0$}, \\
                    0, & \mbox{otherwise.}
                  \end{cases}
   \]
  where $\mu_k$ is the Morse index of $L_k$.
\end{itemize}
The indices $\mu_k$ are even, since each $L_k$ is symmetric and $S^1$-equivariant. Also observe that $\sign\det L=1$ if $V_0=0$ and that $\sign\det L=\sign\det L_0$ if $V_0\ne0$. If $v\in\cO^{S^1}$ is a non-degenerate zero of $f\in\cC^1_{S^1,\nabla}$ then $\sdeg(f,B_\eps(v))=\sdeg(Df(v),B_1(0))$ for $\eps>0$ small. It follows that $d_0^\nabla(f,\cO)=\deg(f,\cO^{S^1},0)$ is the Brouwer degree of $f$ constrained to the set of fixed points of the action of $S^1$. One can also formulate an explicit formula for the degree $\sdeg(f,B_\eps(S^1*v))$ if $S^1*v$ is a non-degenerate orbit of zeroes of $f\in\cC^1_{S^1,\nabla}$, and $v$ is not a fixed point of the action. Since this formula is a bit more complicated and since it is not needed in its full strength we only state the following fact:
\begin{itemize}
  \item[(D6)] Let $f:V\supset\cO\to V$ be in $\cC^1_{S^1,\nabla}$ with $f^{-1}(0)=S^1*v$, and $S^1*v$ is a non-degenerate orbit of zeroes of $f$ with isotropy group $I_v\cong\Z_k$. Then $|d_k^\nabla(f,\cO)|=1$.
\end{itemize}

Now we extend this degree to the infinite-dimensional setting. Let $X$ be a separable Hilbert space with an orthogonal action of the group $S^1$, i.e.\ there is a continuous homomorphism $R:S^1\to\cB(X)$ from $S^1$ into the space of bounded linear operators on $X$ such that each $R(\theta)$ is an orthogonal linear map. The action of $\theta\in S^1$ on $u\in X$ is denoted by $\theta*u:=R(\theta)[u]$. We want to define a degree theory for $S^1$-equivariant maps of the form $\Phi=L-\Psi:\La\to X$ where $L:X\to X$ is a bounded self-adjoint linear operator and $\Psi:\La\to X$ is the gradient of an $S^1$-invariant function defined on an open subset $\La\subset X$. The original extension from \cite{Rybicki:1994} dealt with the case $L=\id$ and $\Psi$ completely continuous. For applications to Hamiltonian systems Rybicki in \cite{Rybicki:2001} considered the case where $L$ is a Fredholm operator of index $0$ and the positive and negative eigenspaces are infinite-dimensional. This implies that $X$ is the form domain of the quadratic form $u\mapsto\ska{Lu,u}$. In our application, however, the functional does not have this property because $\Psi$ is not defined on (an open subset of) the form domain of the quadratic form.

We consider the following class of operators. Let $L\in\cB(X)$ be a bounded, self-adjoint linear operator on $X$. We assume that there is a Hilbert space decomposition
\[
X = \clos\left(\bigoplus_{k\in{\N_0}} E_k\right),\quad E_j\perp E_k\text{ for }j\ne k,
\]
such that the following conditions hold.
\begin{itemize}
  \item[(A1)] $E_k$ is a finite-dimensional, $S^1$-invariant linear subspace of $X$, and the isotropy group of $u\in E_{k}\setminus\{0\}$ is $\Z_k$ for $k\in\N$.
\end{itemize}
Thus $E_k$ is the isotypical component of $E_0^\perp$ corresponding to $\R^2[k]$. For $n\in\N_0$ we set $X_n:=\bigoplus_{k=0}^n E_k$ and write $P_n:X\to X_n$ for the orthogonal projection, so that $P_n[u]\to u$ as $n\to\infty$ for every $u\in X$. The above decomposition is adapted to $L$ in the sense:
\begin{itemize}
  \item[(A2)] $E_0=\Kern(L)$, and for each $k\ne0$: $L(E_k)=E_k$. 
  \item[(A3)] The map $L+P_0$ defines an isomorphism $X\to Y$ onto a Hilbert space $Y\leq X$.
\end{itemize}
In our application $X=H^1(\R/2\pi\Z,\R^{2N})$, the spaces $E_k$ correspond to the $k$-th Fourier modes, $L$ is the $H^1$-gradient of the quadratic form $\frac12\int_0^{2\pi}\langle M_\Ga\dot{u},J_Nu\rangle$ on $X$, and $Y=H^2(\R/2\pi\Z,\R^{2N})$. Recall that the form domain of this quadratic form is $H^{1/2}(\R/2\pi\Z,\R^{2N})$.

Concerning the nonlinear map $\Psi$ we assume:
\begin{itemize}
  \item[(A4)] $\Psi:\ov{\cO}\rightarrow X$ is the gradient of an $S^1$-invariant $C^1$-function $\fK:\Lambda\rightarrow \R$ restricted to the closure of an open, bounded and invariant set $\ov{\cO}\subset\Lambda$.
  \item[(A5)] The image of $\Psi$ is contained in $Y$ and the set $(L+P_0)^{-1}\circ\Psi\big(\ov{\cO}\big)$ is relatively compact in $X$.
\end{itemize}

\begin{Lem}\label{lem:deg-1}
  If (A1)-(A5) hold, and if the equation $Lu-\Psi(u)=0$ does not have a solution $u\in\pa\cO$ then there exists $n_0\in\N$ such that the equation $Lu-P_{n_0}\Psi(u)-t\left(P_n-P_{n_0}\right)\Psi(u)=0$ does not have a solution $u\in X_n\cap\pa\cO$ for $n\ge n_0$, $t\in[0,1]$.
\end{Lem}

\begin{proof}
Arguing by contradiction, suppose there exist sequences $t_k\in[0,1]$ and $u_k\in X_{n_k}\cap\pa\cO$ with $n_k\ge k$ such that $Lu_k-P_k\Psi(u_k)-t_k\left(P_{n_k}-P_k\right)\Psi(u_k)=0$ for all $k\in\N$. Then
\[
u_k-(L+P_0)^{-1}[P_0u_k]-(L+P_0)^{-1}\left[P_k\Psi(u_k)+t_k\left(P_{n_k}-P_k\right)\Psi(u_k)\right] = 0
\quad\text{for all }k.
\]
After passing to subsequences we may assume that $t_k\to t\in[0,1]$, $P_0u_k\to v$ because $X_0=E_0$ is finite-dimensional, and $(L+P_0)^{-1}[\Psi(u_k)] \to w$ by (A5). Then
\[
(L+P_0)^{-1}[P_k\Psi(u_k)] = P_k\circ(L+P_0)^{-1}[\Psi(u_k)] \to w
\]
and similarly $(L+P_0)^{-1}\left[\left(P_{n_k}-P_k\right)\Psi(u_k)\right]\to0$. It follows that $u_k \to u:=v+w \in\pa\cO$, and $Lu-\Psi(u)=0$, a contradiction.
\end{proof}

Lemma~\ref{lem:deg-1} and (D2)--(D5) imply that
\[
\begin{aligned}
&\sdeg(L-P_n\Psi,\cO\cap X_n)+(d_0^\nabla(L-\Psi,\cO),0,0,\ldots) =\\
&\hspace{1cm}
  \sdeg(L-P_{n_0}\Psi,\cO\cap X_{n_0}) + d_0^\nabla(L-\Psi,\cO)\cdot\sdeg(L,X_n\cap (X_{n_0})^\perp).
\end{aligned}
\]
Recall that $d_0^\nabla(L-\Psi,\cO)=\deg(L-\Psi,\cO^{S^1},0)$ is the Brouwer degree of $L-\Psi$ constrained to the fixed point set. As a consequence of our discussion the number
\[
\sdeg(L-P_n\Psi,\cO\cap X_n)-\deg(L-\Psi,\cO^{S^1},0)\cdot\sdeg(L+P_0,X_n)
\]
is independent of $n\ge n_0$ with $n_0$ from Lemma~\ref{lem:deg-1}. Therefore we can define:

\begin{Def}\label{def:deg}\rm
  For a bounded, self-adjoint linear operator $L\in\cB(X)$ and $\Psi:\ov{\cO}\to X$ such that (A1)-(A5) hold the degree for $S^1$-equivariant gradient maps is defined as
  \[
  \sdeg(L-\Psi,\cO) = \left(d_k^\nabla(L-\Psi,\cO)\right)_{k\in\N_0}\in\bigoplus_{k=0}^\infty \Z,
  \]
  where $d_0^\nabla(L-\Psi,\cO)=\deg(L-\Psi,\cO^{S^1},0)$ and for $k\neq 0$
\[
	d_k^\nabla(L-\Psi,\cO)=
	\lim_{n\rightarrow\infty}\left(d_k^\nabla(L-P_n\Psi,\cO\cap X_n)-\deg(L-\Psi,\cO^{S^1},0)\cdot d_k^\nabla(L+P_0,X_n)\right).
\]
\end{Def}

It is a standard argument to prove that $\sdeg$ has the properties (D1)-(D4) with $V$ replaced by $X$ and $f$ replaced by $L-\Psi$ satisfying (A1)-(A5). The same is valid for property (D6) provided the non-degenerate orbit of zeroes $S^1*v$ is contained in a finite-dimensional subspace $X_n\leq X$.

\begin{Rem}\rm
  A somewhat different approach would be to pass from $L[u]-\Psi(u)=0$ to $F(u)=u-(L+P_0)^{-1}[\Psi(u)]=0$. Then $F$ is a compact perturbation of the identity but not a gradient. It is also not $S^1$-orthogonal in the sense of \cite{Rybicki:1994}, a generalization of gradient maps. Consequently the degree from \cite{Rybicki:1994} still cannot be used, and one needs to develop a new version.
\end{Rem}

Now we state a continuation theorem suitable for our application. We consider a family of equations of the form
\beq[par]
Lu-\Psi(r,u) = 0 \qquad (r,u)\in\cD\subset \R^+\times X.
\eeq
Here $S^1$ acts trivially on $\R$ and $L\in\cB(X)$ is a bounded, self-adjoint linear operator on $X$ as above. Concerning the nonlinear map $\Psi$ we require:
\begin{itemize}
 \item[(A6)] $\Psi:\cD\to X$ is defined on an open and invariant subset $\cD\subset \R^+\times X$, it is continuous, equivariant, and $\Psi(r,.)$ is the gradient of $\fK(r,\cdot)$, where $\fK:\cD\rightarrow \R$ is $S^1$-invariant, continuous and differentiable with respect to the $u$ component.
  \item[(A7)] The image of $\Psi$ is contained in $Y$. If $B\subset\R\times X$ is bounded, closed, and $B\subset\cD$, then the set $(L+P_0)^{-1}\circ\Psi(B)$ is relatively compact in $X$.
\end{itemize}

The set of solutions of \eqref{eq:par} will be denoted by $\cS:=\set{(r,u)\in\cD: Lu-\Psi(r,u)=0}$. Observe that if $B\subset\R\times X$ is $S^1$-invariant, closed, bounded and satisfies $B\subset\cD$ then $\cS\cap B$ is compact. This follows easily from (A7). For $M\subset\R^+\times X$ and $r\in\R^+$ we use the notation $M_r=\{u\in X:(r,u)\in M\}$.

\begin{Thm}\label{thm:cont}
  Suppose (A1)-(A3), (A6), (A7) hold, and suppose there exist $r_0>0$ and a relatively open, $S^1$-invariant subset $\cU\subset(0,r_0]\times X$ such that:
  \begin{itemize}
  \item[(i)] For every $r\in(0,r_0]$: $\cU_r\ne\emptyset$, bounded, $\overline{\cU}_r\subset\cD_r$.
  \item[(ii)] $\cS\cap \pa\cU=\emptyset$ where $\pa\cU$ is the relative boundary of $\cU$ in $(0,r_0]\times X$.
  \end{itemize}
  If $\sdeg(L-\Psi(r_0,\,.\,), \cU_{r_0})\ne0$ then there exists a connected component $\cC\subset\cS$ with the following properties:
   \begin{itemize}
  \item[a)] $(\cC\cap\cU)_r\ne\emptyset$ for every $r\in(0,r_0]$,
  \item[b)] $\cC\setminus \cU$ is not contained in a compact subset of $\cD$.
  \end{itemize}
  \end{Thm}

\begin{proof}
We first add two points at infinity to the set $\cD\setminus\pa\cU$:
\[
\cD^*:=(\cD\setminus\pa\cU)\cup\set{\infty_1,\infty_2}.
\]
In order to define the topology of $\cD^*$ we set for $0<\eps<1$:
$$
\cD(\eps)=\set{(r,u)\in\cD : r\in[\eps,\eps^{-1}],\dist(u,\pa\cD_r)\geq \eps,\norm{u}\leq \eps^{-1}}.
$$
A neighborhood basis of $\infty_1$ is given by the family $(\{\infty_1\}\cup \cU)\setminus \cD(1/n)$, $n\in\N$, and a neighborhood basis of $\infty_2$ is given by $(\{\infty_2\}\cup (\cD\setminus\ov\cU))\setminus \cD(1/n)$, $n\in\N$. Then $\cD^*$ is a normal topological space, and $\cS^*:=\cS\cup\{\infty_1,\infty_2\}$ is a compact subspace of $\cD^*$. We need to prove that there exists a connected set $\cC\subset\cS$ such that $\infty_1,\infty_2\in\ov\cC\subset\cD^*$. According to \cite[Proposition 5]{Alexander:1981}, a refinement of Whyburn's lemma, it is sufficient to show that $\infty_1$ and $\infty_2$ are not separated in $\cS^*$. Arguing by contradiction suppose that there exist two open subsets $V_1,V_2\subset\cD^*$ such that $V_1\cap V_2=\emptyset$, $\infty_1\in V_1$, $\infty_2\in V_2$, and $\cS^*\subset V_1\cup V_2$. Then
\[
V_1 \subset \{\infty_1\}\cup\cU\cup\opint(\cD(\eps)) \quad\text{and}\quad
V_2 \subset \{\infty_2\}\cup\cD\setminus\clos(\cU\setminus\cD(\eps))
\]
for some $0<\eps<\min\{1,r_0\}$. It follows that
\[
\begin{aligned}
&\sdeg(L-\Psi(r_0,.),(V_1\cap \cU)_{r_0}) + \sdeg(L-\Psi(r_0,.),(V_1\setminus\ov{\cU})_{r_0})\\
 &\hspace{1cm}= \sdeg(L-\Psi(r_0,.),(V_1)_{r_0})
 = \sdeg(L-\Psi(1/\eps,.),(V_1)_{1/\eps}) = 0
\end{aligned}
 \]
and
\[
\sdeg(L-\Psi(r_0,.),(V_1\setminus\ov{\cU})_{r_0}) = \sdeg(L-\Psi(\eps,.),(V_1\setminus\ov{\cU})_{\eps}) = 0,
\]
hence
\[
\sdeg(L-\Psi(r_0,.),(V_1\cap \cU)_{r_0}) = 0.
\]
Moreover we have
\[
\sdeg(L-\Psi(r_0,.),(V_2\cap \cU)_{r_0}) = \sdeg(L-\Psi(\eps,.),(V_2\cap \cU)_{\eps}) = 0.
\]
This leads to the contradiction
\[
\begin{aligned}
0 &\ne\sdeg(L-\Psi(r_0,.),\cU_{r_0})\\
  &= \sdeg(L-\Psi(r_0,.),(V_1\cap \cU)_{r_0}) + \sdeg(L-\Psi(r_0,.),(V_2\cap \cU)_{r_0}) = 0.
\end{aligned}
\]
\end{proof}

\section{The functional setting}\label{sec:setting}
From now on we assume without loss of generality that $a_0=0$. We want to find solutions of \eqref{eq:Om} that are close to the solution $rZ(t/r^2)$ of \eqref{eq:R2} for $r>0$ small. Since $r=0$ is a singular limit for this ansatz we make a blow-up argument. Fixing $r>0$ and setting $u(s)=\frac1r z(r^2s)$ we see that $z$ solves $(HS)$ if and only if $u$ solves
\beq[Hr]
\quad \Ga_k\dot{u}_k = J\nabla_{u_k} H_r(u),\quad k=1,\dots,N,
\eeq
with
\[
H_r(u) := H(ru)+\frac{1}{2\pi}\sum_{j\neq k}\Ga_j\Ga_k\log r+F(0)
 = H_0(u)-F(ru)+F(0).
\]
Clearly $H_r(u)\to H_0(u)$ as $r\to 0$. The Hamiltonian $H_r$ is defined on
$$
\cO_r=\{u\in\R^{2N}: r u_k\in\Om, u_j\neq u_k\text{ for all } j\neq k\}.
$$
Observe that $\cO_r=\cF_N(\frac1r\Om)$ for $r>0$, and $\cO_0=\cF_N(\R^2)$.

Recall from Section~\ref{sec:results} the space $X = H^1(\R/2\pi\Z,\R^{2N})$ and the fixed point subspace $X^\ga := \set{u\in X: \ga*u=u}$ for $\ga\in\Sym(H)$. We shall seek $2\pi$-periodic solutions $u\in X^\ga$ of \eqref{eq:Hr}, corresponding to $2\pi r^2$-periodic solutions $z$ of \eqref{eq:Om}. Solutions of \eqref{eq:Hr} with period $2\pi$ are critical points of the corresponding action functional. In order to define this functional let
$$
\La = \set{(r,u)\in \R\times X: u(t)\in \cO_r\text{ for all }t\in\R },
$$
and, for $r\in\R$,
$$
\La_r = \set{u\in X:(r,u)\in\La}.
$$
Clearly $\La$ is an open subset of $\R\times X$, and $\La_r$ is open in $X$. Now the action functional $\fJ:\La\rightarrow\R$ is defined by
$$
\fJ(r,u) = \fJ_r(u) =
 \frac12\int_0^{2\pi}\ska{M_\Ga\dot{u},J_Nu}_{\R^{2N}}\:dt-\int_0^{2\pi}H_r(u)\:dt
$$
$\fJ$ is of class $\cC^2$ and critical points of $\fJ$ are solutions of \eqref{eq:Hr}. Observe that
\[
\fJ_r(u) = \fJ_0(u) - \int_0^{2\pi}F(ru)\:dt + 2\pi F(0).
\]
We want to show that the gradient $\Phi_r:=\nabla\fJ_r$ has the form suitable for our degree theory. The decomposition of $X$ is given by the Fourier modes, of course. For $k\in\Z$ we define
\[
B_k:\R\to SO(2N),\quad B_k(t):=e^{-kJ_Nt},
\]
and
\[
E_k := \set{B_k\al+B_{-k}\beta: \al,\beta\in\R^{2N}} \subset X.
\]
Observe that $\fJ_0(u)$ has the form
\beq[form-fJ]
\fJ_0(u) = \frac12\ska{Lu,u}-\int_0^{2\pi}H_0(u)\:dt
\eeq
with $L:X\to X$ given by $Lu = (\id-\De)^{-1}(-J_NM_\Ga\dot{u})$. Here $\De u = \ddot{u}$ defines an isomorphism $\De:H^{s+2}\cap(E_0)^\perp\to H^s\cap (E_0)^\perp$ for any $s\ge0$ where
\[
H^s=H^s(\R/2\pi\Z,\R^{2N})
 = \set{\sum_{k\in\Z}B_k\alpha_k\in L^2(\R/2\pi\Z,\R^{2N}): \sum_{k\in\Z}\abs{k}^{2s}\abs{\alpha_k}^2<\infty}.
\]
The operator $L\in\cB(X)$ is a bounded self-adjoint linear operator with range
\[
\Range(L) \subset Y = H^2 \Subset X = H^1.
\]
Clearly $E_0=\Kern(L)$, $L(E_k)=E_k$ for $k\ne0$, and $L+P_0$ defines an isomorphism $X \cong Y$ where $P_0:X\to E_0$ is the orthogonal projection.

The nonlinearity $\fK(r,u)=\fK_r(u)$ defined by
\[
\fK_r:\La_r\to\R,\quad \fK_r(u) = \int_0^{2\pi}H_r(u)\:dt = \int_0^{2\pi}H_0(u)\:dt+\int_0^{2\pi}F(ru)\:dt-2\pi F(0)
\]
is in $\cC^2(\La)$ and $\Psi_r:=\nabla\fK_r:\La_r\to X$ is given by
\[
\Psi_r(u)=(\id-\De)^{-1}[\nabla H_0(u)]+r(\id-\De)^{-1}[\nabla F(ru)]
\]
Note that $\Psi_r(u)\in H^3\Subset Y$, hence  $(L+P_0)^{-1}\circ\Psi$ maps bounded subsets of $\La$ that are also closed in $\R\times X$ to relatively compact subsets of $X$. Thus we see that $\Phi=\nabla_u\fJ = L-\Psi:\La\to X$ satisfies (A1)-(A3), (A6), (A7).

Next for $\ga\in\Sym(H)$ we set
\[
\La^\ga := \La\cap(\R\times X^\ga) \quad\text{and}\quad
\La_r^\ga := \La_r\cap X^\ga.
\]
Since $\Phi_r$ is equivariant with respect to $\ga$, it induces a restriction $\Phi^\ga:\La^\ga\to X^\ga$. Thus it remains to find a continuum $\cC^\ga=\cC^\ga(a_0,Z)\subset\La^\ga$ of solutions $(r,u)\in\La^\ga$ of the equation $\Phi^\ga_r(u)=0$ with the properties stated in Theorem~\ref{thm:main1}. This will be a consequence of the continuation theorem \ref{thm:cont}.

\section{A degree computation}\label{sec:deg-comp}
We fix $\ga\in \Sym(H)$ and a relative equilibrium $Z\in X^\ga$ of \eqref{eq:R2} as in \eqref{eq:rel-equilib} with minimal period $2\pi$ and assume that $Z$ is $\ga$-non-degenerate. We also assume that $a_0=0$ is a  stable critical point of the Robin function $h$. Recall the notation $\wh{a}=(a,\ldots,a)\in\R^{2N}$ for $a\in\R^2$ and the space $D:=\{\wh{a}:a\in\R^2\}\subset\R^{2N}\subset X^\ga$. The space $D\oplus\R\dot{Z}$ is a subset of the kernel of the linearization $D\Phi^\ga_0(Z)$ because $\Phi_0=\nabla\fJ_0$ and $\fJ_0$ is invariant under translations and under the action of $S^1$.
The $\ga$-nondegeneracy of $Z$ means that the differential $D\Phi^\ga_0(Z):X^\ga\to X^\ga$ has kernel $D\oplus\R\dot{Z}$. Since $(L+P_0)^{-1}\circ D\Phi_0(Z)=\id-P_0-(L+P_0)^{-1}\circ (\id-\De)^{-1}\circ H''_0(Z)$ is a compact perturbation of identity one sees that $Z$ is an isolated zero of $\Phi_0$ restricted to $N_Z:=(D\oplus \R\dot{Z})^\perp\subset X^\ga$. I.e. there exists $0<\de<\norm{Z}$ so that the following holds:
\beq[def-delta]
  u\in N_Z,\ \norm{u-Z}\le\de,\ \Phi_0(u)=0\qquad\Longrightarrow\qquad u=Z.
\eeq
Thus if
\[
N_\de:=S^1*\set{u\in N_Z:\norm{u-Z}<\de}\cap D^\perp=B_{\de}(S^1*Z)\cap D^\perp
\]
denotes the open $\de$-neighborhood of $S^1*Z$ in $D^\perp=\R\dot{Z}\oplus N_Z$, then $\Phi_0$ does not have zeroes in the closure of $N_\de$ except the orbit $S^1*Z$. There also exists $\eps_0>0$ so that $a_0=0$ is the only zero of $\nabla h:\Om\to\R$ in the closed disc $\ov{B_{\eps_0}(0)}\subset\Om$.

The main result of this section is the following proposition.

\begin{Prop}\label{prop:deg}
  Suppose $Z$ is $\ga$-non-degenerate and $a_0=0$ is an isolated zero of $\nabla h$ with index
  $\deg(\nabla h,B_{\eps_0}(0),0)\ne0$, and $\sum_{k=1}^N\Ga_k\neq 0$. Then there exists $r_0>0$ and there exists a relatively open, $S^1$-invariant subset $\cU\subset\La\cap\big((0,r_0]\times X^\ga\big)$ with the following properties:
  \begin{itemize}
  \item[(i)] $(0,r_0]\times (S^1*Z) \subset \cU \subset (0,r_0]\times\set{\wh{b}+v:\wh{b}\in D,\ v\in N_\de}$
  \item[(ii)] $\ov{\cU}\cap\big((0,r_0]\times X^\ga\big)\subset\La^\ga$
  \item[(iii)] $\Phi(r,u)\ne0$ if $(r,u)\in\pa\cU$ where $\pa\cU$ is the relative boundary of $\cU$ in $(0,r_0]\times X^\ga$.
  \item[(iv)] For any sequence $(r_n,u_n)\in\cU$ with $r_n\to0$ there holds $r_nP_D[u_n]\to0$.
  \item[(v)] For $0<r\le r_0$ the set $\cU_r:=\set{u\in X^\ga:(r,u)\in\cU}$ is bounded and $\sdeg(\Phi_r,\cU_r)\ne0$; more precisely there holds $d_1^\nabla(\Phi_r,\cU_r)\ne0$.
\end{itemize}
\end{Prop}

In order to prove Proposition~\ref{prop:deg} we consider the homotopy $h:[0,1]\times\cD\to X^\ga$ defined by
\[
\begin{aligned}
  h(t,r,u) &:= L[u] - (\id-\De)^{-1}[\nabla H_0(u)] + (1-t)r(\id-\De)^{-1}[\nabla F(ru)]\\
           &\hspace{1cm} + trP_D\circ(\id-\De)^{-1}\big[\nabla F(rP_D[u])\big]
\end{aligned}
\]
where $P_D:X^\ga\to D$ is the orthogonal projection and $\cD:=\set{(r,u)\in\La^\ga:rP_D[u]\in\Omega^N}$. Observe that there exists $r_1>0$ such that
\[
\set{(r,u)\in (0,r_1]\times X^\gamma: \norm{P_D[u]}\leq \eps_0/r,\ u-P_D[u]\in\ov{N_\delta}\subset D^\perp\cap \Lambda^\gamma_0}\subset \cD.
\]
Note further that $h(t,r,.)$ is the gradient of an $S^1$-invariant function and $h(0,r,u)=\Phi_r(u)$.

\begin{Lem}\label{lem:homotopy}
  For every $0<\eps\le\eps_0$ there exists $0<r(\eps)\leq r_1$ with the following property: $h(t,r,u)\ne0$ for all $t\in[0,1]$, all $0<r\le r(\eps)$, and all $u\in A_{\eps,r}\cup B_{\eps,r}$ with
  \[
  A_{\eps,r}=\set{u\in X^\ga:\eps/r\le\norm{P_D[u]}\le\eps_0/r,\ u-P_D[u]\in N_\de}
  \]
  and
  \[
  B_{\eps,r}=\set{u\in X^\ga:\norm{P_D[u]}\le\eps_0/r,\ u-P_D[u]\in\pa N_\de}.
  \]
\end{Lem}

\begin{proof}
Arguing by contradiction we assume that there exists $\eps\in(0,\eps_0]$, and sequences $0<r_n\to0$, $t_n\in[0,1]$, $u_n\in A_{\eps,r_n}\cup B_{\eps,r_n}$, such that $h(t_n,r_n,u_n)=0$. Let $v_n:=u_n-P_D[u_n]\in \ov{N_\de}\subset D^\perp$. Then, along a subsequence, $t_n\to t^*\in[0,1]$, $r_nP_D[u_n]\to \wh{c}\in D$ with $\norm{\wh{c}}\leq \eps_0$, $r_nv_n\to 0$ and thus $r_nu_n=r_nv_n+r_nP_D[u_n]\to \wh{c}$.
Since $H_0$ is invariant under translations with elements of $D$ and $P_D\circ (\id-\Delta)^{-1}=P_D$, we obtain from $h(t_n,r_n,u_n)=0$:
\[
	0=\frac{1}{r_n}P_D[h(t_n,r_n,u_n)]=(1-t_n)P_D[\nabla F(r_nu_n)]+t_nP_D[\nabla F(r_nP_D[u_n])]\to P_D[\nabla F(\wh{c})].
\]
A direct computation shows that
\[
\nabla_{z_j} F(\wh{c}) = \Ga_j\sum_{k=1}^N \Ga_k\nabla h(c)
\]
and thus
\[
0=P_D\big[\nabla F(\wh{c})\big] = \frac{1}N\left(\sum_{k=1}^{N}\Ga_k\right)^2\wh{\nabla h(c)}.
\]
By our assumption that $0$ is the only critical point of $h$ in $\ov{B_{\eps_0}(0)}$ we conclude $c=0$, hence $r_nu_n\to0$ and therefore $u_n\in B_{\eps,r_n}$, i.e. $v_n=u_n-P_D[u_n]\in\pa N_\de$.

Applying now $(L+P_0)^{-1}$ to the equation $h(t_n,r_n,u_n)=0$ and using again the invariance of $H_0$ under translations leads to
\[
\begin{aligned}
  0 &= u_n - P_0[u_n] - (L+P_0)^{-1}\circ(\id-\De)^{-1}[\nabla H_0(u_n)]+o(1)\\
    &= v_n - P_0[v_n] - (L+P_0)^{-1}\circ(\id-\De)^{-1}[\nabla H_0(v_n)]+o(1),
\end{aligned}
\]
which implies $v_n\to w\in\pa N_\de$ along a subsequence, due to the fact that $(v_n)_n\subset \pa N_\de$ is bounded and $P_0$, $(L+P_0)^{-1}\circ(\id-\De)^{-1}:X^\ga\to X^\ga$ are compact operators. Therefore we obtain
\[
\begin{aligned}
0 &= h(t_n,r_n,u_n) = L[u_n] - (\id-\De)^{-1}[\nabla H_0(u_n)] + o(1) \\
&= L[v_n] - (\id-\De)^{-1}[\nabla H_0(v_n)] + o(1)\to L[w] - (\id-\De)^{-1}[\nabla H_0(w)] = \Phi_0(w)
\end{aligned}
\]
contradicting the fact that $\Phi_0$ does not have zeroes in $\pa N_\de$.
\end{proof}

\begin{altproof}{Proposition~\ref{prop:deg}}
Using the notation from Lemma~\ref{lem:homotopy} we set $\eps_n:=\eps_0/2^n$, $r_0=r(\eps_0)$, and $r_n:=\min\set{r(\eps_n),r_{n-1}/2}$ for $n\ge1$. Now we define the set $\cU\subset\La\cap\big((0,r_0]\times X^\ga\big)$ as follows. For $r_{n+1}<r\le r_n$ let
\[
\cU_r:=\set{u\in X^\tau:\norm{P_D[u]}<\eps_n/r,\ u-P_D[u]\in N_\de}
\]
and define
\[
\cU:=\set{(r,u):0<r\le r_0,\ u\in\cU_r}\subset \cD.
\]
The properties (i)-(iv) of Proposition~\ref{prop:deg} are immediate consequences of the construction of $\cU$ and Lemma~\ref{lem:homotopy}. It also follows from Lemma~\ref{lem:homotopy} that for $0<r\le r_0$ the degree $\sdeg(\Phi_r,\cU_r)$ is well defined and equal to the degree $\sdeg(h(1,r,\,.\,),\cU_r)$ where
\[
h(1,r,u)=\Phi_0(u) + rP_D\circ(\id-\De)^{-1}[\nabla F(rP_D[u])]=\Phi_0(u) + rP_D[\nabla F(rP_D[u])].
\]
It remains to prove that $d_1^\nabla(h(1,r,\,.\,),\cU_r)\ne0$. Since $\fJ$ is invariant under translations with elements from $D$ it follows that $\Phi_0(u)=\nabla\fJ_0(u)\in D^\perp$, and $\Phi_0(u+\wh{b})=\Phi_0(u)$ for all $u\in N_\de$, all $\wh{b}\in D$. The second summand in the definition of $h(1,r,\,.\,)$ clearly satisfies $rP_D\big[\nabla F(rP_D[u])\big]\in D$. Hence the map $h(1,r,\,.\,)$ can be written as a product
\[
D^\perp\times D\supset N_\de\times D_r \to D^\perp\times D,\quad
 \big(v,\wh{b}\big) \mapsto \big(\Phi_0(v),rP_D\big[\nabla F(r\wh{b})\big]\big),
\]
where
\[
D_r:=\set{\wh{b}\in D: \|\wh{b}\|<\eps_0/r}.
\]
Therefore we can apply the multiplicativity property (D4) of $\sdeg$. In order to do this we first observe that
\[
d_1^\nabla\big(\Phi_0|_{D^\perp},N_\de\big) \ne 0
\]
as a consequence of (D6) because $S^1*Z\subset X_1$ is a non-degenerate orbit of zeroes of the restricition $\Phi_0|_{D^\perp}:\La_0^\ga\cap D^\perp\to D^\perp$, and because $N_\de$ does not contain other zeroes of $\Phi_0$ according to \eqref{eq:def-delta}.
Next we compute the degree $\sdeg(\psi_r,D_r)$ where
\[
\psi_r(\wh{b}) = rP_D\big[\nabla F(r\wh{b})\big]=\frac{r}N\left(\sum_{k=1}^{N}\Ga_k\right)^2\wh{\nabla h(rb)} \in D.
\]
Since $S^1$ acts trivially on $D$ only the component $d_0^\nabla(\psi_r,D_r)$ may be different from zero. This is simply the Brouwer degree $\deg\big(rP_D\big[\nabla F(r\,\,.\,)\big],D_r,0\big)=\deg(\nabla h,B_{\eps_0}(0),0)\ne 0$.
Now the multiplicativity property (D4) yields
\[
d^\nabla_1(h(1,r,\,.\,),\cU_r)
 = d_1^\nabla\big(\Phi_0|_{D^\perp},N_\de\big)\cdot d_0^\nabla(\psi_r,D_r) \ne 0.
\]
\end{altproof}

\section{Proof of Theorem~\ref{thm:main2}}\label{sec:proof}
As a consequence of Proposition~\ref{prop:deg} we can apply Theorem~\ref{thm:cont}. This gives a continuum $\cC=\cC^-\cup\cC^+\subset\R^+\times X^\ga$ of solutions $(r,u)$ of the equation $\Phi_r(u)=0$. Property a) of Theorem~\ref{thm:main1} holds by construction. The neighborhood $\cU$ in property b) of Theorem~\ref{thm:main1} is, of course, the one constructed in Proposition~\ref{prop:deg}. Given a sequence $(r_n,u_n)$ as in Theorem~\ref{thm:main1}~b) we claim that $r_nP_D[u_n]\to0$ and $v_n\to S^1*Z$ as $n\to\infty$. The first claim follows from Proposition~\ref{prop:deg}~(iv). For the second claim we apply $(L+P_0)^{-1}$ to the equation $\Phi_{r_n}(u_n)=0$ and obtain
\[
\begin{aligned}
0 &= (L+P_0)^{-1}\big(\Phi_{r_n}(u_n)\big)\\
  &= u_n - P_0[u_n] - (L+P_0)^{-1}\circ(\id-\De)^{-1}\big[\nabla H_0(u_n)\big]+o(1)\\
  &= v_n - P_0[v_n] - (L+P_0)^{-1}\circ(\id-\De)^{-1}\big[\nabla H_0(v_n)\big]+o(1).
\end{aligned}
\]
As in the proof of Lemma~\ref{lem:homotopy} one sees that $v_n \to v\in \ov{N_\de}\subset D^\perp$ along a subsequence, and $v$ solves $\Phi_0(v)=0$. This implies $v\in S^1*Z$ as claimed, so property b) of Theorem~\ref{thm:main1} holds. Next property c) of Theorem~\ref{thm:main1} corresponds to Theorem~\ref{thm:cont}~a). Property d) is a consequence of the fact that $\cC^+$ is not contained in a compact subset of $\La^\ga$, and lemma \ref{lem:homotopy}.

It remains to proof e). Therefore assume that $a_0=0$ is a non-degenerate critical point of $h$. We consider the auxiliary map $\phi:\Lambda^\gamma\rightarrow X^\gamma$ defined by
\[
\phi(r,u)
 = \begin{cases}
	 (\id-P_D)[\Phi_r(u)]+\frac{1}{r^2}P_D[\Phi_r(u)],&r\neq 0\\
	 \Phi_0(u)+P_D[F''(0)u],&r=0.
   \end{cases}
\]
This has the same zeroes as $\Phi$ in $\La^\ga\setminus(\{0\}\times X^\ga)$. Since $H_0$ is invariant under translations there holds $\phi(r,u)=\Phi_0(u)+r(\id-P_D)\circ (\id-\Delta)^{-1}[\nabla F(ru)]+\frac{1}{r}P_D[\nabla F(ru)]$  for $r\neq 0$. Thus we deduce that $\phi$ is continuous because $F$ is $\cC^2$ and $\nabla F(0)=0$. Observe also that $\phi$ is differentiable with respect to $u$ and that $D_u\phi$ is continuous. We have $\phi(0,Z)=0$ and the $u$-derivative at $(0,Z)$ is given by $T:=D_u\phi(0,Z)=D\Phi_0(Z)+P_D\circ F''(0)$. Again the invariance of $H_0$ under translations implies $v\in\Kern(T)$ if and only if $v\in D\oplus \R\dot{Z}$ and $P_D[F''(0)v]=0$. However, for $v=\hat{a}+\lambda\dot{Z}$ there holds:
\[
P_D[F''(0)v]=P_D[F''(0)\hat{a}]=\frac{1}N\left(\sum_{k=1}^{N}\Ga_k\right)^2\wh{ h''(0)a},
\]
hence $\Kern(T)=\R\dot{Z}$. Now $(L+P_0)^{-1}\circ T$ induces an isomorphism between the Banach spaces $(\R\dot{Z})^\perp$ and $\Range((L+P_0)^{-1}\circ T)=(\R (L+P_0)\dot{Z})^\perp=:R$. Therefore e) follows from \ref{thm:main1} b), \ref{thm:main1} c) and the implicit function theorem applied to the map
\[
P_R\circ(L+P_0)^{-1}\circ \phi:\Lambda^\gamma\cap(\R\times(\R\dot{Z})^\perp)\rightarrow R,
\]
making $r_0$ smaller if necessary. Here $P_R:X^\gamma\rightarrow R$ is the orthogonal projection.

\section{Examples}\label{sec:ex}
Let $Z(t)=e^{-\om J_N t}z$, $z\in\R^{2N}$ fix, be a rigidly rotating solution of \eqref{eq:R2}. In order to prove that $Z$ is non-degenerate we need to consider the so called stability matrix
$$
A = J_N(M_\Ga^{-1}H_0''(z)+\om\cdot\id)\in\R^{2N\times 2N}.
$$
Then according to Remark~\ref{rem:main1}~c) $Z$ is a non-degenerate relative equilibrium provided the linear system
\beq[R2linrot]
\dot{w}=Aw
\eeq
has only 3 linear independent $\frac{2\pi}{\abs{\om}}$-periodic solutions. In order to check this for concrete examples we shall use results of Roberts \cite{Roberts:2013}, who studied the linear stability of relative equilibria and therefore investigated the spectrum of $A$. For the convenience of the reader we recall Lemma~2.4 and some consequences from \cite{Roberts:2013}. For $v\in\R^{2N}$ we use the notation $E_v:=\span\set{v,J_Nv}\subset\R^{2N}$.

\begin{Lem}\label{lem:roberts}
\begin{itemize}
\item[a)] Let $\wh{e_1},\wh{e_2}\in D$ be the standard basis of $D\subset\R^{2N}$. The spaces $E_z$ and $D$ are invariant subspaces of $A$. The representation of $A$ in the basis $(z,J_Nz,\wh{e_1},J_N\wh{e_1})$ of the direct sum $E_z\oplus D$ is given by
    \[
    A=\begin{pmatrix}
    0&0&0&0\\
    2\om&0&0&0\\
    0&0&0&-\om\\
    0&0&\om&0
    \end{pmatrix}.
    \]
\item[b)] Suppose $v$ is a real eigenvector of $M_\Ga^{-1}H_0''(z)$ with eigenvalue $\mu$. Then $E_v$ is an invariant subspace of $A$, on which $A$ is represented by
    \[
    \begin{pmatrix}
    0& \mu-\om\\
    \mu+\om&0
    \end{pmatrix}.
    \]
\item[c)] Suppose $v=v_1+iv_2$ is a complex eigenvector of $M_\Ga^{-1}H_0''(z)$ with eigenvalue $\mu=\xi+i\eta$. Then $\span\set{v_1,v_2,J_Nv_1,J_Nv_2}\subset\R^{2N}$ is a real invariant subspace of $A$, on which $A$ is represented by
    $$
    \begin{pmatrix}
    0 & 0 & \xi-\om & \eta\\
    0 & 0 & -\eta & \xi-\om\\
    \xi+\om & \eta & 0 & 0\\
    -\eta & \xi+\om & 0 & 0
    \end{pmatrix}.
    $$
\end{itemize}
\end{Lem}

Note that the Hamiltonian in \cite{Roberts:2013} differs by a factor of $\pi^{-1}$ from $H_0$ but the corresponding stability matrices coincide, when translating the solution of one system to the other.

\begin{Ex}\label{ex:N=2}
\rm
Let $N=2$ and $\Ga_1,\Ga_2\ne0$ with $\Ga:=\Ga_1+\Ga_2\ne0$. Any initial position $z_1,z_2$ of the two point vortices gives a relative equilibrium solution of \eqref{eq:R2} (see e.g.\ \cite{Newton:2001}). Via translation we can assume that they rotate rigidly around the origin with frequency $\om=\frac{\Ga}{\pi|z_1-z_2|^2}\ne0$. Due to Lemma~\ref{lem:roberts} the stability matrix $A\in\R^{4\times 4}$ of any such solution is given (in a suitable basis) by
\[
A=\begin{pmatrix}
0&0&0&0\\
2\om&0&0&0\\
0&0&0&-\om\\
0&0&\om&0
\end{pmatrix}.
\]
The linear system \eqref{eq:R2linrot} then possesses exactly $3$ linearly independent $\frac{2\pi}{\abs{\om}}$-periodic solutions. This explains Example \ref{ex:2vortices}.
\end{Ex}

\begin{Ex}\label{ex:N=3}
\rm
Now we consider $N=3$ vortices with vortex strengths, $\Ga_1,\Ga_2,\Ga_3\ne0$, and such that $\Ga:=\Ga_1+\Ga_2+\Ga_3\ne0$. Then every equilateral triangle configuration $z_1,z_2,z_3$ is a relative equilibrium solution of the $3$-vortex problem \eqref{eq:R2} (see \cite[Section~2.2]{Newton:2001}). Let $Z(t)=e^{-\om J_3 t}z$ be an equilateral triangle configuration rotating around the origin. The corresponding stability matrix $A$ is a $6\times6$ matrix. In \cite{Roberts:2013} Roberts computed its eigenvalues explicitly in the case when $\om=\Ga/3$; this can always be achieved by a suitable scaling. He showed that in addition to the eigenvalues $0,0,\pm i\om$ of the block in \ref{lem:roberts}a) there are two more eigenvalues given by $\pm\sqrt{\frac{-L}{3}}$, where $L=\Ga_1\Ga_2+\Ga_1\Ga_3+\Ga_2\Ga_3$ is the total vortex angular momentum. Hence the linear system \eqref{eq:R2linrot} has more than 3 linearly independent $\frac{2\pi}{\abs{\om}}$-periodic solutions if $L>0$ and $\sqrt{L/3}\in\om\Z=\frac{\Ga}{3}\Z$, hence if there exists $k\in\Z$ with
\[
3L = k^2\Ga^2 = k^2\left(\Ga_1^2+\Ga_2^2+\Ga_3^2+2L\right).
\]
This is only possible if $k^2=1$ and $L=\Ga_1^2+\Ga_2^2+\Ga_3^2$. Therefore the equilateral triangle configuration is non-degenerate provided $\Ga\ne0$, $L\ne0$ and $L\ne\Ga_1^2+\Ga_2^2+\Ga_3^2$. This result, which is independent of the particular equilateral triangle configuration considered in \cite{Roberts:2013}, has been stated in Example \ref{ex:3vortices}.
\end{Ex}

\begin{Ex}\label{ex:NGon-2}
\rm
Here we consider the case of $N$ identical vortices and assume without loss of generality $\Ga_1=\ldots=\Ga_N=1$. Placing these on the vertices of a regular $N$-gon one obtains a relative equilibrium solution of \eqref{eq:R2}, the famous solution of Thomson. Setting $Z_1(t):=e^{-Jt}e_1\in\R^2$ and $Z_k(t):=Z(t+\frac{2(k-1)\pi}{N})$ for $k=2,\dots,N$, the function $Z=(Z_1,\dots,Z_N)$ solves \eqref{eq:R2}. Note that this solution is a choreography, i.e.\ for $\si:=(1~2~\ldots~N)\in\Sigma_N$ and $\ga:=(\si,\frac{2\pi}{N})\in\Si_N\times S^1$ one has $\ga*Z=Z$. In \cite[Lemma-4.1]{Bartsch-Dai:2016} it was proved that the kernel of $\fJ''_0(Z)$ in $X^\ga$ has dimension $3$, hence $Z$ is $\ga$-non-degenerate..
\end{Ex}

{\sc Address of the authors:}\\[1em]
\parbox{8cm}{Thomas Bartsch, Bj\"orn Gebhard\\
 Mathematisches Institut\\
 Universit\"at Giessen\\
 Arndtstr.\ 2\\
 35392 Giessen\\
 Germany\\
 Thomas.Bartsch@math.uni-giessen.de\\
 Bjoern.Gebhard@math.uni-giessen.de}

\end{document}